\documentclass[11pt]{amsart}
\usepackage{amssymb}
\usepackage{amsthm,amsfonts}

\usepackage{amsfonts}
\usepackage{amsmath}
\usepackage[abbrev]{amsrefs}
\usepackage{enumerate}
\usepackage{color}

\newtheorem{Theorem}{Theorem}[section]

\newtheorem{Lemma}[Theorem]{Lemma}
\newtheorem{Corollary}[Theorem]{Corollary}
\newtheorem{Proposition}[Theorem]{Proposition}
\newtheorem{Definition}[Theorem]{Definition}

\theoremstyle{definition}
\newtheorem{Remark}[Theorem]{Remark}

\DeclareMathOperator{\supp}{supp}

\begin{document}

\title{Diameter two properties  and the Radon-Nikod\'ym property in Orlicz spaces}

\author{Anna Kami\'{n}ska}
\address{Department of Mathematical Sciences,
The University of Memphis, TN 38152-3240}
\email{kaminska@memphis.edu}

\author{Han Ju Lee}
\address{Department of Mathematics Education, Dongguk University - Seoul, 04620 (Seoul), Republic of Korea}
\email{hanjulee@dgu.ac.kr}

\author{Hyung Joon Tag}
\address{Department of Mathematical Sciences,
The University of Memphis, TN 38152-3240}
\email{hjtag4@gmail.com}

\date{\today}

\thanks{
The second author was supported by Basic Science Research Program through the National Research Foundation of Korea(NRF) funded by the Ministry of Education, Science and Technology [NRF-2020R1A2C1A01010377].}

\begin{abstract}
Some necessary and sufficient conditions are found for Banach function lattices to have the Radon-Nikod\'ym property. Consequently it is shown that an Orlicz function space $L_\varphi$  over a non-atomic $\sigma$-finite measure space $(\Omega, \Sigma,\mu)$, not necessarily separable, has the Radon-Nikod\'ym property if and only if $\varphi$ is an $N$-function at infinity and satisfies the appropriate $\Delta_2$ condition. For an Orlicz sequence space $\ell_\varphi$, it has the Radon-Nikod\'ym property if and only if $\varphi$ satisfies  the  $\Delta_2^0$ condition. In the second part  a relationship between uniformly $\ell_1^2$ points of the unit sphere of a Banach space and the diameter of the slices are studied. Using these results, a quick proof is given that an Orlicz space $L_\varphi$ has the Daugavet property only if $\varphi$ is linear, so when $L_\varphi$ is isometric to $L_1$. Another consequence is that  Orlicz spaces equipped with the Orlicz norm generated by $N$-functions never have  the local diameter two property, while it is well-known that when equipped with the Luxemburg norm, it may have that property.  Finally, it is shown that the local diameter two property, the diameter two property, and the strong diameter two property are equivalent in Orlicz function and sequence spaces with the Luxemburg norm under appropriate   conditions on $\varphi$.  	
\end{abstract}

\maketitle

\begin{center}{Dedicated to the memory of Professor W.A.J. Luxemburg}\end{center}

\section{Introduction}

The objective of this paper is to study geometrical properties in real Banach spaces, in particular in Banach function spaces and  Orlicz spaces.
A Banach space $(X, \|\cdot\|)$ is said to have the {\it Daugavet property} if every rank one operator $T: X\to X$ satisfies the equation
\[
\|I + T\| = 1 + \|T\|.
\]

It is well-known that $C[0,1]$ has the Daugavet property. Also, a rearrangement invariant space $X$ over a finite non-atomic measure space with the Fatou property satisfies the Daugavet property if the space is isometrically isomorphic to either $L_1$ or $L_{\infty}$ \cite{AKM,  AKM2}. If a rearrangement invariant space $X$ over an infinite non-atomic measure space is uniformly monotone, then  it is isometrically isomorphic to $L_1$ \cite{AKM2}. Furthermore, the only separable rearrangement invariant space over $[0,1]$ with the Daugavet property is $L_1[0,1]$ with the standard $L_1$-norm \cite{KMMW}. In \cite{KK}, a characterization of Musielak-Orlicz spaces with the Daugavet property has been provided.  We refer to \cite{KSSW, W} for further information on the Daugavet property.

Let $S_X$ and $B_X$ be the unit sphere and the unit ball of a Banach space $X$ and let $X^*$  be the dual space of $X$. A slice of $B_X$ determined by $x^*\in S_{X^*}$ and $\epsilon>0$ is defined by the set
\[
S(x^*; \epsilon) = \{x\in  B_X : x^*(x) > 1 - \epsilon\}.
\]
Analogously, for $x\in S_X$ and $\epsilon >0$, a weak$^*$-slice $S(x, \epsilon)$ of $B_{X^*}$ is defined by the set
\[
S(x; \epsilon) = \{x^*\in B_{X^*}: x^*(x) > 1 - \epsilon\}.
\]
There are several geometrical properties related to slices and weak$^*$-slices. We say that $X$ has

\begin{enumerate}[{\rm(i)}]
\item the {\it local diameter two property} (LD2P) if every slice of $B_X$ has the diameter two.
\item the {\it diameter two property} (D2P) if every non-empty relatively weakly open subset of $B_X$ has the diameter two.
\item the {\it strong diameter two property} (SD2P) if every finite convex combination of slices of $B_X$ has the diameter two.
\item the {\it weak$^*$-local diameter two property} (weak$^{*}$-LD2P) if every weak$^*$-slice of  $B_{X^*}$ has the diameter two.
\item the {\it Radon-Nikod\'ym property} (RNP) if there exist slices of $B_X$ with arbitrarily small diameter.
\end{enumerate}

A few remarks are in order now. Condition $\rm(v)$ is a geometrical interpretation of the classical Radon-Nikod\'ym  property \cite[Theorem 3, p. 202]{DU}. By the  definitions, we  see that properties (i), (ii) and (iii) are on the opposite spectrum of $\rm(v)$. It is clear that $\rm(ii) \implies \rm(i)$, and the implication  $\rm(iii) \implies \rm(ii)$ results from
 \cite[Lemma II.1 p. 26]{GGMS}.
 It is also well-known that  these three properties are not equivalent in general {\cite{ALN, BLR}, see also introductory remarks in \cite{HLP}. A Banach space $X$ with the Daugavet property satisfies the SD2P \cite[Theorem 4.4]{ALN}.

After Preliminaries, in Section 3, we show first that if a Banach function space $X$ over a $\sigma$-finite measure space has the RNP then it must be order continuous. The opposite implication is not true in general. However we prove it  under additional assumptions when  $X$ satisfies the Fatou property, and when  the subspaces of order continuous elements and the closure of simple functions coincide in  its  K\"othe dual space $X'$. We will provide some examples to see that this assumption is necessary in order to show the converse.  Applying the obtained results further, we conclude the section with the necessary and sufficient condition for the RNP in Orlicz spaces. There is a well-known criterion for the RNP in Orlicz spaces $L_\varphi$ over a separable complete non-atomic measure space $(\Omega, \Sigma, \mu)$, generated by an $N$-function $\varphi$. Here we drop the assumption of separability of a measure space and show that necessary and sufficient  conditions for the RNP is that $\varphi$  satisfies appropriate $\Delta_2$ condition and that $\varphi$ is an $N$-function at infinity. In sequence spaces $\ell_\varphi$ we drop the assumption that $\varphi$ is an $N$-function.

In section 4 the Daugavet and various diameter two properties are studied.
In the first main theorem we give a local characterization of uniformly $\ell_1^2$ points $x\in S_X$, where $X$ is a Banach space, and the diameter of the weak$^*$-slice $S(x;\epsilon)$ generated by $x$. Analogously we describe a relationship  between $x^*\in S_{X^*}$ and the diameter of the slice $S(x^*,\epsilon)$ given by $x^*$. Consequently, we obtain a description of global properties of $X$ or $X^*$ being locally octahedral and $X^*$ or $X$ respectively, having (weak$^*$) local diameter two property. We also obtain relationships among the Daugavet property of $X$ and the D2P of $X$ and weak$^*$-LD2P of $X^*$. In Theorem \ref{th:KamKub} we provide sufficient conditions for the existence of uniformly non-$\ell_1^2$ points in $L_\varphi$ and $\ell_\varphi$ both equipped with the Luxemburg norms.  Combining this with the previous general facts we recover instantly that the only Orlicz space $L_\varphi$ generated by a finite function $\varphi$ with the Daugavet property must coincide with $L_1$ as sets with  equivalent norm. The other consequences are that   large class of the Orlicz spaces $L_\varphi^0$ and $\ell_\varphi^0$ equipped with the Orlicz norm, does not have the LD2P, in the striking opposition  to the same Orlicz spaces equipped with the Luxemburg norm.
In the final result we show that the LD2P, D2P, SD2P and the appropriate condition $\Delta_2$ are equivalent in $L_\varphi$ and $\ell_\varphi$.

\section{Preliminaries}
Let $(\Omega, \Sigma, \mu)$ be a measure space with a $\sigma$-finite complete measure $\mu$ and let $L^0(\Omega)$ be the set of all equivalence classes of $\mu$-measurable functions $f:\Omega\to \mathbb{R}$ modulo a.e. equivalence. We denote by $L^0=L^0(\Omega)$ if $\Omega$ is a non-atomic measure space and  by $ \ell^0=L^0(\mathbb{N})$ if $\Omega = \mathbb{N}$ with the counting measure $\mu$. That is, $\ell^0$ consists of all real-valued sequences $x = \{x(n)\}$. A Banach space $(X, \|\cdot\|) \subset L^0(\Omega)$ is called a {\it Banach function lattice} if for $f\in L^0(\Omega)$ and $g \in X$, $0 \leq f \leq g$ implies $f \in X$ and $\|f\| \leq \|g\|$. We call $X$ a {\it Banach function space} if $\Omega$ is a non-atomic measure space and a {\it Banach sequence space} if $\Omega = \mathbb{N}$ with the counting measure $\mu$. A Banach function lattice $(X, \| \cdot \|)$ is said to have the {\it Fatou property} if whenever a sequence $(f_n) \subset X$ satisfies $\sup_n \|f_n\| < \infty$ and $f_n \uparrow f\in L^0(\Omega)$ a.e., we have $f \in X$ and $\|f_n\| \uparrow \|f\|$. An element $f\in X$ is said to be {\it order continuous} if for every $(f_n) \subset L^0(\Omega)$ such that $0\le f_n \le f$, $f_n \downarrow 0$ a.e. implies $\|f_n\| \downarrow 0$. The set of all order continuous elements in $X$ is denoted by $X_a$, and the closure in $X$ of all simple functions belonging to $X$ is denoted by $X_b$. If $X=X_a$ then we say that $X$ is order continuous. In this paper, a simple function is a finitely many valued function whose support is of finite measure. It is well-known that $X_a \subset X_b$ \cite[Theorem 3.11, p. 18]{BS}\cite{ Lux}.

The {\it K\"{o}the dual space}, denoted by $X^{\prime}$, of a Banach function lattice $X$ is a set of $x \in L^0(\Omega)$, such that
\[
\|x\|_{X^{\prime}}= \sup\left\{\int_\Omega xy : \|y\| \leq 1\right\} < \infty.
\]
The space  $X^{\prime}$, equipped with the norm $\| \cdot \|_{X^{\prime}}$, is a Banach function lattice satisfying the Fatou property. It is well known that $X = X''$ if and only if $X$ satisfies the Fatou property \cite[Theorem 1, Ch. 15, p. 71]{Z}.

We say $f,g \in X$ are {\it equimeasurable}, denoted by $f \sim g$, if $\mu\{t: |f(t)| > \lambda\} = \mu\{t: |g(t)| > \lambda\}$ for every $\lambda > 0$. A Banach function lattice $(X, \|\cdot\|)$ is said to be {\it rearrangement invariant} (r.i.) if $f \sim g$ ($f,g \in X$) implies $\|f\| = \|g\|$. The Lebesgue, Orlicz and Lorentz spaces are  classical examples of r.i. spaces. The fundamental function of a r.i. space $X$ over a non-atomic measure space is defined by $\phi_X(t) = \|\chi_{E}\|_X$, $t\ge 0$, where $E\in \Sigma$ is such that $\mu(E) = t$. It is known that $\lim_{t \rightarrow 0^+} \phi_X(t) = 0$ if and only if $X_a = X_b$ \cite[Theorem 5.5, p. 67]{BS}. For the purely atomic case where each atom has the same measure,}$X_a = X_b$ is always true \cite[Theorem 5.4, p. 67]{BS}.

Recall that a measure space $(\Omega, \Sigma, \mu)$ is said to be {\it separable} if there is a countable family $\mathcal{T}$ of measurable subsets such that for given $\epsilon>0$ and for each $E\in \Sigma$ of finite measure there is $A\in \mathcal{T}$ such that $\mu(A\Delta E)<\epsilon$, where $A\Delta E$ is the symmetric difference of $A$ and $E$. It is easy to check that if $\Sigma$ is a $\sigma$-algebra generated by countable subsets then $(\Omega, \Sigma, \mu)$ is separable \cite{Hal}.

A function $\varphi:\mathbb{R}_+\to [0,\infty]$ is called an {\it Orlicz function} if $\varphi$ is convex,  $\varphi(0)=0$, and $\varphi$ is left-continuous, not identically zero nor infinite on $(0,\infty)$.  The complementary function $\varphi_{\ast}$ to $\varphi$ is defined by
\[
\varphi_{\ast}(u) = \sup_{v \geq 0}\{uv- \varphi(v) \}, \ \ \ u\ge 0.
\]
The complementary function $\varphi_*$ is also an Orlicz function and $\varphi_{**} = \varphi$.

An Orlicz function $\varphi$ is an {\it $N$-function at zero} if $\lim_{u \rightarrow 0^+} \frac{\varphi(u)}{u} = 0$ and {\it at infinity} if $\lim_{u \rightarrow \infty} \frac{\varphi(u)}{u} = \infty$. If $\varphi$ is an $N$-function at both zero and infinity  then we say that $\varphi$ is  an {\it $N$-function}. A function $\varphi$ is an $N$-function if and only if $\varphi_*$ is an $N$-function.

An Orlicz function $\varphi$ satisfies the {\it $\Delta_2$ condition} if there exists $K>2$ such that $\varphi(2u) \leq K \varphi(u)$ for all $u\geq 0$,  the $\Delta_2^\infty$ condition if there exist $K>2$ and $u_0\ge 0$  such that $\varphi(u_0) < \infty$ and for all $u\geq u_0$, $\varphi(2u) \leq K \varphi(u)$, and  the $\Delta_2^0$ condition if there exist $K>2$ and $u_0$ such that $\infty > \varphi(u_0) > 0$  and for all $0\le u \leq u_0$, $\varphi(2u) \leq K \varphi(u)$. When we use the term {\it the appropriate $\Delta_2$ condition}, it means $\Delta_2$ in the case of a non-atomic measure $\mu$ with $\mu(\Omega) = \infty$,  $\Delta_2^\infty$ for  a non-atomic measure $\mu$ with $\mu(\Omega) < \infty$, and $\Delta_2^0$ for $\Omega = \mathbb{N}$ with the counting measure i.e. $\mu\{n\} =1$ for every $n\in \mathbb{N}$.

The Orlicz space $L_\varphi(\Omega)$  is a collection of all $f\in L^0(\Omega)$ such that for some $\lambda > 0$,
\[
I_\varphi(\lambda f):= \int_\Omega\varphi(\lambda |f(t)|)\,d\mu(t) = \int_\Omega\varphi(\lambda |f|)\,d\mu  < \infty.
\]
The Orlicz spaces are equipped with  either the Luxemburg norm
\[
\|f\|_\varphi= \inf\left\{\epsilon > 0: I_\varphi\left(\frac{f}{\epsilon}\right) \le 1\right\},
\]
or the Orlicz (or Amemiya) norm
\[
\|f\|_\varphi^0= \sup\left\{\int_\Omega fg : I_{\varphi_*} (g)\le 1\right\} = \inf_{k>0} \frac{1}{k}(1 + I_\varphi(kf)).
\]
It is well-known that $\|f\|_\varphi \le \|f\|_\varphi^0 \le 2\|f\|_\varphi$ for $f\in L_\varphi(\Omega)$.
By $L_\varphi(\Omega)$ we denote an Orlicz space equipped with the Luxemburg norm and by $L_\varphi^0(\Omega)$ with the Orlicz norm.
The  Orlicz spaces with either norms are rearrangement invariant spaces and have the Fatou property.

If $\varphi$ is finite, i.e. $\varphi(u)< \infty$ for all $u>0$, then $(L_\varphi(\Omega))_a \ne \{0\}$ and it contains all simple functions. Therefore
\[
(L_\varphi(\Omega))_a = (L_\varphi(\Omega))_b = \{x\in L^0: I_\varphi(\lambda x) < \infty \ \ \text{for all} \ \ \lambda > 0\}.
\]
It is also well-known that $L_\varphi(\Omega) = (L_\varphi(\Omega))_a$ if and only if  $\varphi$ satisfies the appropriate $\Delta_2$ condition. The K\"othe duals of $L_\varphi(\Omega)$ and $L_\varphi^0(\Omega)$ are described  by Orlicz spaces induced by $\varphi_*$ \cite{Chen, BS}. In fact,
\[
 (L_\varphi(\Omega))' = L_{\varphi_*}^0(\Omega)\ \ \text{ and} \ \ \
(L_\varphi^0(\Omega))' = L_{\varphi_*}(\Omega).
\]

In the case of non-atomic measure  (resp., counting measure), we use the symbols $L_\varphi$ and $L_\varphi^0$, (resp., $\ell_\varphi$ and $\ell_{\varphi}^0$)  for Orlicz  spaces equipped with the Luxemburg and the Orlicz norm, respectively. For complete information on Orlicz spaces we refer the reader to the monographs \cite{BS, Chen, KR, LT1, LT2, Lux}.

\section{The Radon-Nikod\'ym property}

We start with a general result on the Radon-Nikod\'ym property in  Banach function spaces.
\begin{Theorem}\label{th:RNKothe}
Let $X$ be a Banach function space  over a complete $\sigma$-finite measure space $(\Omega, \Sigma, \mu)$.
\begin{itemize}
\item[(i)]
If $X$ has the RNP then $X$ is order continuous.
\item[(ii)]
Assume that $X$ has the Fatou property and $(X')_a = (X')_b$.  Then if $X$ is order continuous then $X$ has the RNP.
\end{itemize}
\end{Theorem}

\begin{proof}
(i) If $X$ is not order continuous then it contains an order isomorphic copy of $\ell_\infty$ \cite[Theorem 14.4, p.220]{AB}. Since $\ell_\infty$ does not have the RNP,  $X$ does not have this property either.

(ii) Suppose that $(X')_a = (X')_b$ and that $X$ is order continuous with the Fatou property. It is well-known that every separable dual space possesses the RNP \cite{DU}.
Since $((X')_a)' = ((X')_b)'$,  $((X')_a)'=(X')' = X''$ and $((X')_a)^*\simeq  ((X')_a)'$ by Corollary 1.4.2 in \cite{BS}. It follows by  the Fatou property that $X'' = X$.
Therefore
\[
((X')_a)' \simeq ((X')_a)^* \simeq X ''=X.
\]
Hence $X$ is the  dual space of $(X')_a$. If the measure space $(\Omega, \Sigma,\mu)$ is separable, then the order continuous space $X$ is also separable by Theorem 2.5.5 in \cite{BS}. Thus in this case, $X$ has the RNP.

Now, suppose that $(\Omega, \Sigma,\mu)$ is not separable and we show that $X$ still has the RNP.  We will use the fact that  a Banach space $X$ satisfies the RNP if and only if every separable closed subspace  $Y\subset X$ has the RNP  \cite[Theorem 2, p. 81]{DU}.

Since $X$ is order continuous  $X=X_a=X_b$. Let $Y\subset X$ be a closed separable subspace of $X$. Then there exists a dense and countable set $\mathcal{Y} \subset Y$. For every $y\in \mathcal{Y} \subset X=X_b$, there exists a sequence of simple functions $(y_n) \subset X$ with supports of finite measure and  such that $\|y-y_n\|_X \to 0$. Each $y_n$ can be expressed as $y_n = \sum_{i=1}^{m_n} a_i^{(n)} \chi_{A_i^{(n)}}$, where $a_i^{(n)}\in \mathbb{R}$, $A_i^{(n)} \in \Sigma$ with $\mu(A_i^{(n)}) < \infty$, so $y \in \overline{span} \{\chi_{A_i^{(n)}}, i=1,\dots,m_n, \ n\in \mathbb{N}\}$. Letting  $\mathcal{A}_y = \{A_i^{(n)}: i=1,\dots,m_n, n\in\mathbb{N}\}$ and $\mathcal{A} = \cup_{ y\in \mathcal{Y}}\mathcal{A}_y$, the family $\mathcal{A}$  is  countable.

For our convenience, let $\mathcal{A} = \{E_i: i\in \mathbb{N}\}$.  For each $i\in \mathbb{N}$ we have $\mu(E_i) < \infty$. Then we have
\[
Y = \overline{\mathcal{Y}}\subset \overline{span} \{\chi_{E_i}, \ E_i \in \mathcal{A}\} \subset X.
\]
Let $\widetilde\Omega = \cup_{i=1}^\infty E_i$, $\sigma(\mathcal{A})$ be the smallest $\sigma$-algebra of $\Omega$ containing $\mathcal{A}$, $\widetilde{\Sigma} = \{\widetilde{\Omega} \cap E: E \in \sigma(\mathcal{A})\}$ and $\widetilde\mu = \mu|_{\widetilde\Sigma}$ the measure $\mu$ restricted to $\widetilde{\Sigma}$. In fact, it is easy to show that $\widetilde{\Sigma} = \sigma(\mathcal{A})$. Hence $\widetilde{\Sigma}$ is generated by a countable set, namely $\mathcal{A}$, so the measure space $(\widetilde\Omega, \widetilde\Sigma,\widetilde\mu)$ is separable  \cite[Theorem B, p. 168]{Hal}). Now  define the set
\[
\widetilde X = \{x\chi_{\widetilde\Omega}: x\in X, \ x  \ \text{is} \  \widetilde{\Sigma} - \text{measurable}\}.
\]
It is straightforward to check that $\widetilde X$ is a closed subspace of $X$ such that it is an order continuous Banach function space on $(\widetilde\Omega,\widetilde\Sigma, \widetilde\mu)$ with the Fatou property. So $\widetilde X$ is separable.   The K\"othe dual of $\widetilde{X}$ is
\[
\widetilde X' := (\widetilde X)' = \{y\chi_{\widetilde\Omega}: y\in X', \ y  \ \text{is} \  \widetilde{\Sigma} - \text{measurable}\}.
\]
Clearly $\widetilde X' \subset L^0(\widetilde\Omega, \widetilde\Sigma, \widetilde\mu)$. From the assumption we have $(\widetilde X')_a = (\widetilde X')_b$. Hence
$\widetilde X = \widetilde{X}'' \simeq ((\widetilde{X}')_a)^*$ by Corollary 1.4.2 in \cite{BS} again. Therefore, $\widetilde X$ is a separable dual space such that $Y \subset \overline{span} \{\chi_{E_i}, \ E_i\in \mathcal{A}\}\subset \widetilde X$, which implies that $\widetilde X$ and hence  $Y$ has the RNP. Since the choice of $Y$ was arbitrary, $X$ has the RNP.
\end{proof}

\begin{Remark}
$(1)$ The Fatou property in (ii) is a necessary assumption. For example, take $X=c_0$. This space does not have the RNP \cite[p. 61]{DU} and clearly does not satisfy the Fatou property.  However, since $(c_0)'= \ell_1$ and $\ell_1$ is order continuous we have $((c_0)')_a = (\ell_1)_a = (\ell_1)_b = ((c_0)')_b$, which is the second assumption of the theorem.

$(2)$ The assumption $(X')_a = (X')_b$ in (ii) is also necessary. Consider $X=L_1[0,1]$ that is clearly order continuous.  Moreover
$(X')_a = (L_\infty[0,1])_a = \{0\}$ and $(X')_b = (L_\infty[0,1])_b = L_\infty[0,1]$. Hence $(X')_a \ne (X')_b$ and it is well-known that $L_1[0,1]$ does not have the RNP \cite[p. 60]{DU}.

\end{Remark}

\begin{Proposition}\label{pro1}
Let $\mu$ be a non-atomic measure and $\varphi$ be a finite  Orlicz function. If $\varphi$ is not an $N$-function at infinity, then $L_\varphi$  contains a subspace isomorphic to $L_1[0,1]$.
\end{Proposition}
\begin{proof}
 Suppose $\varphi$ is not an $N$-function at infinity. We will show that given $A\in\Sigma$  with $\mu(A) < \infty$, the space $L_\varphi(A) = \{x\chi_A: x\in L_\varphi\}$ is equal to $ L_1(A)$ with equivalent norms.   By the assumption $\lim_{u\to\infty} \varphi(u)/u =K< \infty$, and by the fact that
  the function $\varphi(u)/u$ is increasing, there exist $M> 0$ and  $u_0 >0$ such that
 \begin{equation}\label{eq:31}
 \varphi(u) \le Ku \ \ \text{for} \ \ u\ge 0, \ \ \  \text{and} \ \ \ \varphi(u)\ge M u \ \ \ \text{for}  \ \ u \geq u_0.
 \end{equation}
 Let $f\in L_\varphi(A)$ with $\|f\|_\varphi = 1$. Then $\supp f\subset A$ and $I_\varphi(f) \le 1$. Set $A_1 = \{t\in A: |f(t)| \le u_0\}$. Thus in view of the second part of inequality (\ref{eq:31}) we get
 \[
 \|f\|_1 = \int_{A_1} |f| \, d\mu + \int_{A\setminus A_1} |f|\, d\mu \le u_0 \mu(A) + \frac1M I_\varphi(f) \le C,
 \]
 where $C = u_0 \mu(A) + \frac1M.$
 Therefore for any $f$ from $L_\varphi(A)$, $\|f\|_1 \le C \|f\|_\varphi$.

On the other hand, by the second part of inequality (\ref{eq:31}),   for any $E\subset A$ we have
\[
\int_A\varphi\left(\frac{\chi_E}{K\mu(E)}\right)\, d\mu = \int_E\varphi\left(\frac{1}{K\mu(E)}\right) \, d\mu \le 1.
\]
It follows that $\|\chi_E\|_\varphi \le K\mu(E) =K \|\chi_E\|_1$. Hence for any simple function $x = \sum_{i=1}^n a_i \chi_{E_i}\in L_\varphi(A)$ with $E_i \cap E_j = \emptyset$ for $i\ne j$,
\[
\|x\|_\varphi \le \sum_{i=1}^n |a_i| \|\chi_{E_i}\|_\varphi \le K \sum_{i=1}^n |a_i| \mu(E_i) = K \|x\|_1.
\]
Now by the Fatou property of $L_1(A)$ and $L_\varphi(A)$,  $\|x\|_\varphi \le K \|x\|_1$ for every $x \in L_1(A)$.

  Hence $L_\varphi(A) = L_1(A)$ with equivalent norms, and the proof is completed since $L_1(A)$ contains a subspace isomorphic to $L_1[0,1]$  (see \cite[p. 127, Theorem 9 (1)]{Lac}).
\end{proof}

Recall that an Orlicz function is said to be finite if its range does not contain the infinity.
\begin{Lemma}\label{lem:finite}
\rm{(a)}  A finite   Orlicz function  $\varphi$ is an $N$-function at infinity if and only if $\varphi_*$ is finite.

\rm{(b)} Let $\mu$ be a non-atomic measure. If $\varphi$ is a finite Orlicz function  then $\|\chi_A\|_\varphi = 1/\varphi^{-1}(1/t)$, where $t>0$, $\mu(A) = t$. Consequently
\[
\lim_{t\to 0+} \phi_{L_\varphi}(t) = \lim_{t\to 0+}  1/\varphi^{-1}(1/t) = 0.
\]
\end{Lemma}

\begin{proof}
(a) Suppose  $\varphi$ is not an $N$-function at infinity. Then there exists $K>0$ such that for every $u >0$, $\varphi(u) \leq Ku$. Hence
\[
\varphi_*(v) = \sup_{u >0}\{uv - \varphi(u)\} \geq \sup_{u >0}\{(v-K)u\}.
\]
Therefore if $v > K$ then $\varphi_*(v) =\sup_{u>0}\{(v-K)u\} = \infty$.

 Conversely, suppose there exists $K > 0$ such that for every $v > K$, $\varphi_*(v) = \infty$. Then
\[
\varphi(u) = \sup\{uv - \varphi_*(v) : v \in (0,K)\}.
\]
By $\frac{\varphi(u)}{u} = \sup\{v - \frac{\varphi_*(v)}{u}: v \in (0, K)\}$, we have $\lim_{u \rightarrow \infty} \frac{\varphi(u)}{u} \leq K < \infty$, which shows that $\varphi$ is not an $N$-function at infinity.

(b) Let $a_\varphi=\sup\{t: \varphi(t)=0\}$, and let  $A\in \Sigma$, $\mu(A) = t$, $t > 0$. Then $I_\varphi(\chi_A/\epsilon) = 0$ if $\epsilon \ge 1/a_\varphi$, and  $I_\varphi(\chi_A/\epsilon) = \varphi(1/\epsilon) t$ if $ \epsilon < 1/a_\varphi$. By the latter condition if $I_\varphi(\chi_A/\epsilon) = 1$, we get that $\|\chi_A\|_\varphi = \epsilon = 1/\varphi^{-1}(1/t)$. Clearly for $t\to 0+$ we get that $1/\varphi^{-1}(1/t) \to 0$.
\end{proof}

The next result provides a criterion of the Radon-Nikod\'ym property of Orlicz spaces over non-atomic measure spaces. We do not need the assumption of separability of the measure space  \cite[Theorem 3.32]{Chen}.

\begin{Theorem}\label{th:OrRN-funct}
Let $\mu$ be a complete $\sigma$-finite, non-atomic measure on $\Sigma$ and $\varphi$ be a finite Orlicz function. Then the Orlicz spaces $L_\varphi$ {\rm (}and  $L_\varphi^0${\rm )} over
$(\Omega, \Sigma, \mu)$ have the Radon-Nikod\'ym property if and only if $\varphi$ is an $N$-function at infinity and  satisfies the appropriate $\Delta_2$ condition.
\end{Theorem}
\begin{proof}
Since the Luxemburg and Orlicz norms are equivalent we consider only $L_\varphi$ equipped with the Luxemburg norm.	By the assumption that $\varphi$ is an $N$-function at infinity and Lemma \ref{lem:finite}(a)  we get that $\varphi_*$ is finite on $(0,\infty)$.
Applying now Lemma \ref{lem:finite}(b) to the function $\varphi_*$ we get that  $\phi_{L_{\varphi_*}}(t) \to 0$ if $t\to 0+$.  Hence $\lim_{t\to 0+} \phi_{L^0_{\varphi_*}}(t) = 0$.   Applying now Theorem 2.5.5 in \cite{BS} we
get $(L_{\varphi_*}^0)_a = (L_{\varphi_*}^0)_b$ and in view of $(L_\varphi)' = L_{\varphi_*}^0$ \cite[Corollary 8.15, p. 275]{BS} \cite{KR},  we have $((L_\varphi)')_a = ((L_\varphi)')_b$.

 It is well-known that $L_\varphi$ has the Fatou property and that $L_{\varphi}$ is order continuous if and only if $\varphi$ satisfies the appropriate $\Delta_2$ condition. Therefore, by Theorem \ref{th:RNKothe}(ii) the Orlicz space $L_\varphi$ has the RNP.

For the converse, assume that $L_\varphi$ has the RNP. Since $L_1[0,1]$ does not have the RNP, $\varphi$ needs to be an $N$-function at infinity by Proposition~\ref{pro1}. If $\varphi$ does not satisfy the appropriate $\Delta_2$ condition, then $L_\varphi$ is not order continuous, and by Theorem \ref{th:RNKothe}(i) it does not have the RNP.
\end{proof}

By Theorem 2.5.4 in \cite{BS}, $X_a = X_b$ holds for every rearrangement invariant sequence space $X$. Consequently we obtain a characterization of the RNP in  Orlicz sequence spaces $\ell_\varphi$ as a consequence  of  Theorem  \ref{th:RNKothe}. This result is well-known for $\varphi$ being an $N$-function \cite[Theorem 3.32]{Chen}.

\begin{Theorem} \label{th:RNP-ORseq}
Let $\varphi$ be a finite Orlicz function. An Orlicz sequence space $\ell_{\varphi}$ has the Radon-Nikod\'ym property if and only if $\varphi$ satisfies the $\Delta_2^0$ condition.
\end{Theorem}

\begin{proof}

Since any Orlicz sequence space is an r.i. space with the Fatou property, we always have $((\ell_{\varphi})')_a = (\ell_{\varphi_*}^0)_a = (\ell_{\varphi_*}^0)_b = ((\ell_{\varphi})')_b$. Moreover  it is well-known that $\ell_{\varphi}$ is order continuous if and only if $\varphi$ satisfies the $\Delta_2^0$ condition  \cite[Proposition 4.a.4]{LT1}. Hence, $\ell_{\varphi}$ has the RNP by Theorem \ref{th:RNKothe}.

Conversely, suppose that $\ell_{\varphi}$ has the RNP. Then $\ell_{\varphi}$ is order continuous by Theorem \ref{th:RNKothe}. This implies that $\varphi$ satisfies the $\Delta_2^0$ condition.

\end{proof}

\section{Locally octahedral norm, uniformly non-$\ell_1^2$ points, diameter two properties and the Daugavet property}

In this section, we first examine the relationship between locally octahedral norms and the Daugavet property.
\begin{Definition}\cite{G, BLR2, HLP}
A Banach space $X$ is locally octahedral if for every $x \in X$ and $\epsilon >0$, there exists $y \in S_X$ such that $\|\lambda x + y\| \geq (1 - \epsilon) (|\lambda| \|x\| + \|y\|)$ for all $\lambda \in \mathbb{R}$.
\end{Definition}

A point $x\in S_X$ is called a \emph{uniformly non-$\ell_1^2$} point if there exists $\delta>0$ such that $\min\{\|x + y\|, \|x - y\|\} \leq 2-\delta$ for all $y\in S_X$. Motivated by this, we introduce the following.

\begin{Definition}
A point $x\in S_X$ is called a \emph{uniformly $\ell_1^2$} point if, given $\delta>0$, there is $y\in S_X$  such that $\min\{\|x + y\|, \|x - y\|\} > 2-\delta$.
\end{Definition}

By  Proposition 2.1 in \cite{HLP} we get immediately the following corollary.

\begin{Corollary}\label{cor:octah-non}
Every point $x \in S_X$ is a uniformly $\ell_1^2$ point if and only if the Banach space $X$ is locally octahedral.
\end{Corollary}

\begin{Lemma}\cite{HLP}\label{lem:aux}
If $x, y \in S_X$ satisfy $\|x \pm y\| > 2 - \delta$ and $\alpha, \beta \in \mathbb{R}$, then
	\[
	(1-\delta)(|\alpha| + |\beta|) < \|\alpha x \pm \beta y\| \leq |\alpha| + |\beta|.
	\]
\end{Lemma}

\begin{proof} See the proof of the implication  from (iii) to (ii) in Proposition  2.1 in \cite{HLP}.

\end{proof}

In the next theorem we give a local characterization of uniformly $\ell_1^2$ points $x\in S_X$ (resp. $x^*\in S_{X^*}$) and the diameter of the slice $S(x;\epsilon)$ (resp.  the diameter of the weak$^{*}$-slice $S(x^*,\epsilon)$). The techniques used in the proof are somewhat similar to the proof of Theorem 3.1 in \cite{HLP}, but the key ideas are more subtle emphasizing  the local nature of  discussed properties.

It follows a corollary on relationships between  global properties of local diameter two property in $X$ and of $X^*$  being locally octahedral, as well as between the weak$^*$-local diameter two property of $X^*$ and $X$ being locally octahedral.

\begin{Theorem}\label{th:unif}
\rm(a) An element 	$x \in S_{X}$ is a uniformly $\ell_1^2$ point if and only if the diameter of a weak$^{*}$-slice $S(x;\epsilon)$ is two for every $\epsilon > 0$.

\rm(b) An element $x^* \in S_{X^*}$ is a uniformly $\ell_1^2$ point if and only if ${\rm{diam}}\,S(x^*;\epsilon)=2$ for every $\epsilon > 0$.
\end{Theorem}
\begin{proof}
We will prove only (a) since (b) follows analogously. Suppose that for all $0<\epsilon < 1$, ${\rm{diam}} \, S(x; \epsilon) = 2$. Then there exist $x_1^*, x_2^*\in S(x;\epsilon)$ such that
	\begin{equation}\label{eq:11}
	x_1^*(x)> 1-\epsilon, \ \ \ x_2^* (x) > 1 - \epsilon, \ \ \|x_1^* - x_2^*\| > 2- \epsilon.
	\end{equation}
	Hence we can find $y \in S_X$ with $(x_1^* - x_2^*)(y) > 2 - \epsilon$.
Thus
	\[
	2 \ge x_1^*(y) - x_2^*(y) > 2-\epsilon \ \ \ \text{and} \ \ \ x_1^*(y)\le1,\  -x_2^*(y) \le 1,
	\]
	and so $x_1^*(y) > 1-\epsilon$ and $-x_2^*(y) > 1-\epsilon$. Combining this with (\ref{eq:11}) we get that $x_1^*(x+y) > 2-2\epsilon$, $x_2^*(x-y) > 2 - 2\epsilon$, and so $\|x + y\| > 2 - 2\epsilon$ and $\|x - y\| > 2 -2\epsilon$. We showed that for every $0< \epsilon< 1$ there exists $y\in S_X$ such that
	\[
	\min\{\|x + y\|, \|x - y\|\} > 2 -2\epsilon,
	\]
	which means that $x$ is a uniformly $\ell_1^2$ point.
	
	Conversely, suppose that $x \in S_X$ is a uniformly $\ell_1^2$ point. Then for any $\epsilon>0$, there exists $y \in S_X$ such that $\|x \pm y\| > 2 - \epsilon$. Define bounded linear functionals $x_1^*, x_2^*$ on the subspace  ${\rm span}\{x,y\}$ such that
	\[
	x_1^*(x) = 1,\ \ x_1^*(y) = 0,\ \ x_2^*(x) = 0 \ \  \text{and} \ \ x_2^*(y) = 1.
	\]
	\noindent Note that $\|x_1^*\| \geq 1$ and $\|x_2^*\| \geq 1$. By Lemma \ref{lem:aux}, for $\alpha, \beta \in \mathbb{R}$ we have
	\[
	|(x_1^* \pm x_2^*)(\alpha x + \beta y)| = |\alpha \pm \beta| \leq |\alpha| + |\beta| \leq (1-\epsilon)^{-1}\|\alpha x + \beta y\|,
	\]
	
	\noindent so $\|x_1^* \pm x_2^*\| \leq (1-\epsilon)^{-1}$.
	
	Now, let $\widetilde{x_1}^* = \frac{x_1^* + x_2^*}{\|x_1^* + x_2^*\|}$ and $\widetilde{x_2}^* = \frac{x_1^* - x_2^*}{\|x_1^* - x_2^*\|}$. Then
	\[
	\|\widetilde{x_1}^* - (x_1^* + x_2^*)\| = |\|x_1^* + x_2^*\| - 1| \leq \left|\frac{1}{1-\epsilon} - 1 \right| = \frac{\epsilon}{1 - \epsilon}.
	\]
	Similarly,
	\[
	\|\widetilde{x_2}^* - (x_1^* - x_2^*)\| \leq \frac{\epsilon}{1 - \epsilon}.
	\]
	 Since $(x_1^* \pm x_2^*)(x)=1$, we have $\widetilde{x_1}^*(x) = \frac{1}{\|x_1^* + x_2^*\|} \geq 1 - \epsilon$ and $\widetilde{x_2}^*(x) = \frac{1}{\|x_1^* - x_2^*\|} \geq 1 - \epsilon$. Hence $\widetilde{x_1}^*, \widetilde{x_2}^* \in S(x,\epsilon)$. Furthermore,
	\begin{eqnarray*}
	\|\widetilde{x_1}^* - \widetilde{x_2}^*\| &=& \|\widetilde{x_1}^*  + (x_1^* + x_2^*) - (x_1^* + x_2^*) + (x_1^* - x_2^*) - (x_1^* - x_2^*) - \widetilde{x_2}^*\| \\
	&\geq& 2 \|x_2^*\| -\|\widetilde{x_1}^* - (x_1^* + x_2^*)\| -\|\widetilde{x_2}^* - (x_1^* - x_2^*)\| \geq 2 - \frac{2\epsilon}{1-\epsilon}.	
	\end{eqnarray*}
\noindent Since $\epsilon > 0$ is arbitrary, ${\rm{diam}}\, S(x, \epsilon) = 2$. Finally by the Hahn-Banach theorem, we can extend the bounded linear functionals $x_1^*$ and $x_2^*$ from ${\rm span}\{x,y\}$ to $X$ and  the proof is completed.
\end{proof}

Combining Corollary \ref{cor:octah-non} and Theorem \ref{th:unif} we obtain the following result proved earlier in  \cite{HLP}.

\begin{Corollary} \cite[Theorem 3.2, 3.4]{HLP}\label{HLP} Let $X$ be a Banach space. Then the following hold.
	\begin{enumerate}
		\item[$(1)$] $X$ is locally octahedral if and only if $X^*$ satisfies the weak$^{*}$ local diameter two property.
		\item[$(2)$] $X^*$ is locally octahedral if and only if $X$ satisfies the local diameter two property.
	\end{enumerate}
		
\end{Corollary}

Recall the equivalent geometric interpretation of the Daugavet property.

\begin{Lemma}\cite[Lemma 2.2]{KSSW}
	\label{lem:Daug}
	The following are equivalent.
	\begin{enumerate}[{\rm(i)}]
		\item A Banach space $(X,\|\cdot\|)$ has the Daugavet property,
		\item\label{Daugii} For every slice $S = S(x^*,\epsilon)$ where $x^*\in S_{X^*}$, every $x \in S_X$ and every $\epsilon>0$, there exists $y\in S_{X}\cap S$ such that $\|x+y\|>2-\epsilon$,
		\item\label{Daugiii} For every weak$^{*}$-slice $S^* = S(x,\epsilon)$ where $x\in S_{X}$, every $x^* \in S_{X^*}$ and every $\epsilon>0$, there exists $y^*\in S_{X^*}\cap S^*$ such that $\|x^*+y^*\|>2-\epsilon$,
	\end{enumerate}
\end{Lemma}

 The next result is known in a stronger form \cite[Theorem 4.4]{ALN} \cite[Corollary 2.5]{BLR2} \cite{HLP}, namely, if $X$ has the Daugavet property then it has the SD2P as well as its dual $X^*$ has the weak$^*$-SD2P.

\begin{Proposition}\label{prop:slice}
If a Banach space $X$ has the Daugavet property, then $X$ has the local diameter two property and $X^*$ has the weak$^*$-local diameter two property.
\end{Proposition}
\begin{proof} Let $x^*\in S_{X^*}$ and $S(x^*; \epsilon)$ be a slice of $B_X$. Then there is $x\in S_X$ such that $-x^*(x) > 1- \epsilon$. By (iii) of Lemma \ref{lem:Daug} we find $y\in S_X$ with $x^*(y) > 1-\epsilon$ and $\|x + y\| > 2 - 2\epsilon$. Clearly $-x, y \in S(x^*;\epsilon)$, and so ${\rm diam} \, S(x;\epsilon) = 2$.

Now let $x\in S_X$ and $S(x; \epsilon)$ be a weak$^*$-slice of $B_{X^*}$. There exists $y^*\in S_{X^*}$ such that $-y^*(x) > 1 - \epsilon$. By (ii) of Lemma \ref{lem:Daug} there is $x^* \in S_{X^*}$ with $x^*(x) > 1 - \epsilon$ and $\|x^* + y^*\| > 2 - 2\epsilon$. Since both $x^*, -y^* \in S(x;\epsilon)$, $\epsilon > 0$ is arbitrary, we have that ${\rm diam} \, S(x,\epsilon) = 2$.

\end{proof}

The next result is an instant corollary of Theorem \ref{th:unif} and Proposition \ref{prop:slice}.

\begin{Corollary}\cite[Proposition 4.4]{KK}
	\label{prop}
If $(X,\|\cdot\|)$ has the Daugavet property, then  all elements in $S_X$ and $S_{X^*}$ are uniformly $\ell_1^2$ points.
\end{Corollary}

Next we shall consider Orlicz spaces $L_{\varphi}, \,\ell_\varphi$ and $L_{\varphi}^0, \, \ell_\varphi^0$.   Let us define first the following numbers related to  Orlicz function $\varphi:\mathbb{R}_+ \to [0,\infty]$. Recall the Orlicz function $\varphi$  is called  a  {\it linear function} if $\varphi(u)= ku$ on $\mathbb{R}_+$ for some $k>0$. Set

\[
 d_\varphi =\sup\{u: \varphi(u)\ \ \text{is linear}\}, \ \  c_\varphi = \sup\{u: \varphi(u) \le 1\},\ \ \
  b_\varphi = \sup\{u: \varphi(u) < \infty\}.
 \]

\begin{Lemma} \cite[Lemma 4.1]{KK} \label{lem:1}
Let $\varphi$ be an Orlicz function  For every closed and bounded inteval $I \subset (d_{\varphi}, b_{\varphi})$ there is a constant $\sigma \in (0,1)$ such that $2 \varphi(u/2)/\varphi(u) \leq \sigma$ for $u \in I$. Moreover, if $\varphi(b_\varphi) < \infty$ then the same statement holds true for closed intervals $I \subset (d_\varphi, b_\varphi]$.
\end{Lemma}

\begin{Theorem}\label{th:KamKub}
\begin{enumerate}[{\rm(1)}]
\item[\rm(i)] Let $\mu$ be non-atomic. Let $\varphi$ be an Orlicz function such that  $\varphi(b_\varphi)\mu(\Omega) >1$ and $d_\varphi < b_\varphi$. Then there exists $a> 0$ and $A \in \Sigma$ such that $x = a \chi_A, \|x\|_{\varphi} =1$ and $x$ is uniformly non-$\ell_1^2$ point in $L_{\varphi}$.  If $b_\varphi = \infty$ then $x\in (L_\varphi)_a$.

\item[\rm(ii)] Let $\mu$ be the counting measure on $\mathbb{N}$ and $\varphi$ be an Orlicz function such that $d_\varphi < c_\varphi$ and $\varphi(c_\varphi) = 1$. Then there exist $a > 0$ and $A \subset \mathbb{N}$ such that $x = a \chi_A, \|x\|_{\varphi} =1$ and $x$ is uniformly non-$\ell_1^2$ point in $\ell_{\varphi}$. If $b_\varphi = \infty$ then $x\in (\ell_\varphi)_a$.
\end{enumerate}
\end{Theorem}

\begin{proof}

(i): By the assumptions on $\varphi$ and non-atomicity of $\mu$, there exist $A\in \Sigma$ and $a\in (d_\varphi,b_\varphi)$  such that $\varphi(a) \mu(A) = 1$.
 Letting $x=a \chi_A$, we get $I_{\varphi}(x) =1$, and  $\|x\|_{\varphi} = 1$. Clearly $x\in (L_\varphi)_a$ if $b_\varphi = \infty$.

Let $y \in S_{L_\varphi}$ be arbitrary. Hence  for a.e. $t\in \Omega$, $|y(t)| < b_\varphi$ if $b_\varphi=\infty$ or $|y(t)| \le b_\varphi$ if $b_\varphi<\infty$. Then, for any $\lambda >1$, $I_{\varphi}({y}/{\lambda}) \leq 1$. We claim that
\begin{equation}\label{cond1}
\text{there exist }\, d\in (a,b_\varphi) \,\, \text{and} \,\, B=\{t \in \Omega : |y(t)| \leq  d \chi_A (t)\} \,\, \text{such that} \,\, \mu(A \cap B) >0.
\end{equation}

Indeed, let first $b_\varphi = \infty$. Define $B_k = \{t \in \Omega : |y(t)| \leq  k \chi_A (t)\}$ for $k \in  \mathbb{N}$. The sequence of sets  $\{B_k\}$ is increasing, and so
$0 < \mu(A) = \mu (A \cap (\cup_{k=1}^{\infty} B_k) = \lim_{k \rightarrow \infty} \mu (A \cap B_k)$, and this implies that there exists $m \in \mathbb{N}$ such that $2a <m$, $\mu (A \cap B_{m}) >0$. Letting $B = B_{m}$, $d=m$, we get (\ref{cond1}).

Let now $b_\varphi < \infty$. Define $C_k = \{t \in \Omega : |y(t)| \leq  (b_\varphi - 1/k) \chi_A (t)\}$ for $k \in  \mathbb{N}$. Like before, $\{C_k\}$ is increasing and $\lim_{k \rightarrow \infty} \mu(A \cap C_k)>0$. So there exists $m$ such that $b_\varphi - 1/m > a$. Let now $d = b_\varphi - 1/m$ and $B = C_m$, and so (\ref{cond1}) is satisfied.

Set
\begin{equation}\label{eq:111}
 \gamma = I_{\varphi} (a \chi_{A \setminus  B}).
 \end{equation}
  Clearly $\gamma\in [0,1)$.
For any $\delta>0$, there exists $1>\epsilon>0$ such that $I_{\varphi}((1+ \epsilon)x) = \varphi((1+\epsilon)a)\mu(A)  \leq 1 + \delta$. We can choose $\epsilon$ so small that we also have $(1+\epsilon)a < d$.
Let $z = (1+\epsilon) x = (1 + \epsilon)a \chi_A$. Thus
\begin{equation}\label{eq:11}
I_\varphi(z) \le 1+\delta.
\end{equation}
Define
\[
D = \{t \in A \cap B : x(t)y(t) \geq 0 \},\,\, E = (A \cap B) \setminus D.
\]
 For $t \in A \cap B$, $\max \{ |z(t)|, |y(t)| \} = \max \{ |(1+\epsilon)a|, |y(t)| \} \in [a,d]$.
 Since $D \subset A \cap B$, we have $|z(t) - y(t)|/2 \le \max\{|z(t)|, |y(t)|\}$ for $t\in D$. Moreover by Lemma \ref{lem:1}, there exists $\sigma \in (0,1)$ such that $2\varphi(u/2)/\varphi(u) \le \sigma$ for $u\in [a,d] \subset (d_{\varphi}, b_{\varphi})$. Therefore
\begin{eqnarray*}
I_{\varphi}\left(\frac{z-y}{2} \chi_D \right) \leq I_{\varphi}\left(\frac{\max\{|z|, |y| \}}{2} \chi_D \right) &\leq& \frac{\sigma}{2} I_{\varphi}(\max\{|z|, |y| \}\chi_D)\\
&\leq& \frac{\sigma}{2}(I_{\varphi}(z \chi_D)+ I_{\varphi}(y \chi_D)).
\end{eqnarray*}

\noindent Analogously we can also show that
\[
I_{\varphi} \left(\frac{z+y}{2} \chi_E \right) \leq \frac{\sigma}{2}(I_{\varphi}(z \chi_E)+ I_{\varphi}(y \chi_E)).
\]
 Then, by the convexity of $\varphi$ and  $A \cap B = D\cup E$,
\begin{eqnarray*}
&\,&I_{\varphi}\left(\frac{z-y}{2} \chi_{A \cap B} \right)  + I_{\varphi} \left(\frac{z+y}{2} \chi_{A \cap B} \right)\\ &=&  I_{\varphi} \left(\frac{z-y}{2} \chi_D \right) + I_{\varphi} \left(\frac{z+y}{2} \chi_D \right) + I_{\varphi} \left(\frac{z-y}{2} \chi_E \right) + I_{\varphi} \left(\frac{z+y}{2} \chi_E \right)\\
&\leq& \frac{\sigma}{2}(I_{\varphi}(z\chi_D) + I_{\varphi}(y\chi_D)) + \frac{1}{2}(I_{\varphi}(z\chi_D) + I_{\varphi}(y\chi_D))+ \frac{1}{2}(I_{\varphi}(z\chi_E) + I_{\varphi}(y\chi_E)) \\
&+& \frac{\sigma}{2}(I_{\varphi}(z\chi_E) + I_{\varphi}(y\chi_E)
= \frac{1+ \sigma}{2}(I_{\varphi}(z \chi_{A \cap B}) + I_{\varphi}(y \chi_{A \cap B})).
\end{eqnarray*}

Now, choose $\delta \in (0, \frac{(1-\sigma)(1- \gamma)}{2})$.  By the assumption  $I_{\varphi}(y) \leq 1$  and by (\ref{eq:111}), (\ref{eq:11})  we have
\[
2+ \delta \geq I_{\varphi}(y) + 1 + \delta \geq I_{\varphi}(y) + I_{\varphi}(z), 
\]
and so
\begin{eqnarray*}
2+ \delta - I_{\varphi}\left(\frac{z+y}{2} \right) - I_{\varphi}\left(\frac{z-y}{2}\right) &\geq& I_{\varphi}(y)+ I_{\varphi}(z) - I_{\varphi}\left(\frac{z+y}{2}\right) - I_{\varphi}\left(\frac{z-y}{2}\right)\\
&\geq& I_{\varphi}(y)+ I_{\varphi}(z) - \frac{1+ \sigma}{2}(I_{\varphi}(z \chi_{A \cap B}) + I_{\varphi}(y \chi_{A \cap B}))\\
&\geq& \frac{1- \sigma}{2}(I_{\varphi}(z \chi_{A \cap B}) + I_{\varphi}(y \chi_{A \cap B}))\\
&\geq& \frac{1- \sigma}{2}I_{\varphi}(a \chi_{A \cap B}) = \frac{(1- \sigma)(1-\gamma)}{2},
\end{eqnarray*}

\noindent which implies that
\[
I_{\varphi}\left(\frac{z+y}{2} \right) + I_{\varphi}\left(\frac{z-y}{2} \right) \leq 2+ \delta - \frac{(1- \sigma)(1-\gamma)}{2} \le 2.
\]
It follows
\[
\min \left \{I_{\varphi}\left(\frac{z+y}{2} \right), I_{\varphi}\left(\frac{z-y}{2} \right) \right\} \leq 1.
\]
If $I_{\varphi}\left(\frac{z+y}{2}\right) \leq 1$, then $\left \| \frac{z+y}{2} \right \|_{\varphi} \leq 1$,  and so $\left \| \frac{x+(y/(1+\epsilon))}{2} \right \|_{\varphi} \leq \frac{1}{1+\epsilon}$. Moreover,

\begin{equation*}
\left | \left\| \frac{x+y}{2}\right \|_{\varphi} - \left \|\frac{x+(y/(1+\epsilon))}{2} \right \|_{\varphi} \right | \leq \left \| \frac{x+y}{2} - \frac{x+(y/(1+\epsilon))}{2} \right \|_{\varphi} = \frac{\epsilon}{2(1+\epsilon)}.
\end{equation*}
Hence
\[
\left \| \frac{x+y}{2} \right \|_{\varphi} \leq \left \|\frac{x+(y/(1+\epsilon))}{2} \right \|_{\varphi} + \frac{\epsilon}{2(1+\epsilon)} \leq \frac{1}{1+ \epsilon} + \frac{\epsilon}{2(1+ \epsilon)} = 1 - \frac{\epsilon}{2(1+\epsilon)}.
\]
 In a similar way, if $I_{\varphi} \left(\frac{z-y}{2} \right) \leq 1$, then $\left \| \frac{x-y}{2} \right \|_{\varphi} \leq 1 - \frac{\epsilon}{2(1+\epsilon)}$. Thus, we just showed that  for any $y \in S_{L_\varphi}$, $\min \left\{ \left \| \frac{x+y}{2} \right \|_{\varphi}, \left \| \frac{x-y}{2} \right \|_{\varphi} \right \} \leq 1 - \frac{\epsilon}{2(1+\epsilon)}$, which means

\begin{equation*}
\min \{ \| x+y\|_{\varphi}, \| x-y\|_{\varphi}\} \leq 2 - \frac{\epsilon}{1+ \epsilon} < 2.
\end{equation*}
Therefore, $x= a \chi_A$, $\|x\|_{\varphi} = 1$ is a uniformly non-$\ell_1^2$ point in $L_\varphi$.

\vspace{2mm}

(ii): If $x \in S_{\ell_{\varphi}}$, then $I_\varphi(x) = \sum_{i=1}^{\infty} \varphi(|x(i)|) \leq 1$. So for every $i \in \mathbb{N}$, $\varphi(|x(i)|) \leq 1$. Hence for any element of $S_{\ell_\varphi}$, we only consider $u \geq 0$ such that $\varphi(u)\leq 1$. Then by the assumptions $1= \frac{1}{\varphi(c_\varphi)}< \frac{1}{\varphi(d_\varphi)}$, there exist $a \in (d_{\varphi}, c_\varphi]$ and $A \subset \mathbb{N}$ such that $\varphi(a) = 1/\mu(A)$. Let $x = a\chi_A$. Then $I_\varphi(x) = \varphi(a)\mu(A) = 1$ and $\|x\|_\varphi = 1$.  If $b_\varphi = \infty$ then $x \in (\ell_\varphi)_a$.

Now for $y \in S_{\ell_\varphi}$, we want to show that there exists $d \in (a, c_\varphi)$ and $B = \{i \in \mathbb{N} : |y(i)| \leq d\}$ such that $\mu(A \cap B) > 0$, which corresponds to (\ref{cond1}) in function case.

Since $y$ is in the unit ball of $\ell_\varphi$, for each $i\in \mathbb{N}$, $|y(i)| \le c_\varphi$.
 Define $C_k = \{i \in \mathbb{N} : |y(i)| \leq  (c_\varphi - 1/k) \chi_A (i)\}$ for $k \in  \mathbb{N}$. The sequence  $\{C_k\}$ is increasing and
 \[
 0 < \mu (A) = \mu (A \cap (\cup_{k=1}^{\infty} C_k) = \lim_{k \rightarrow \infty} \mu (A \cap C_k).
 \]
  So there exists $m$ such that $c_\varphi - 1/m > a$. Let now $d = c_\varphi - 1/m$ and $B = C_m$. Then
$d\in (a,c_\varphi)$, $|y(i)| \le d \chi_A(i)$ for $i\in A\cap B$ and $\mu(A\cap B) > 0$.

Further we proceed analogously as in the proof for function spaces starting from  (\ref{eq:111}). We apply Lemma \ref{lem:1} for the interval $I=[a,d] \subset (d_\varphi, c_\varphi)\subset (d_\varphi, b_\varphi)$.
\end{proof}

Concerning the Daugavet property we will consider only the case of non-atomic measure since it is not difficult to show that any rearrangement invariant sequence space never has the Daugavet property.
In \cite{AKM2}, it was shown that an Orlicz space $L_\varphi$ generated by a finite Orlicz function $\varphi$  has the Daugavet property if and only if the space is isometrically isomorphic to $L_1$.  Similar result
 can be derived also from \cite{KK} where it was proved for  Musielak-Orlicz spaces. Below the given proof for Orlicz spaces $L_\varphi$  is   much simpler than those in \cite{AKM2, KK}.  In fact it is a direct corollary of Theorem \ref{th:KamKub}(i).

\begin{Theorem}\label{thm:DaugOrlicz}
	Let $\mu$ be a non-atomic measure. 	
	If $\varphi$ is a finite   Orlicz function then the only Orlicz space $L_\varphi$ having the Daugavet property
correspond to a linear function $\varphi$, that is $L_\varphi = L_1$ isometrically.
	\end{Theorem}

\begin{proof}

If $L_\varphi = L_1$ isometrically, clearly the Orlicz space has the Daugavet property.

Supposing $L_\varphi$ has the Daugavet property, by Corollary \ref{prop}, every point of the unit sphere of $L_\varphi$ is a uniformly $\ell_1^2$ point. Applying now Theorem \ref{th:KamKub}(i),  $d_\varphi = b_\varphi$, where $b_\varphi = \infty$ by the assumption that $\varphi$ assumes finite values. Therefore $\varphi(u) = ku$, for some $k>0$ and all $u\ge 0$. Consequently, $L_\varphi = L_1$ and $\|\cdot\|_\varphi = k\|\cdot\|_1$.
\end{proof}

\begin{Theorem}\label{thm:Orlicznormdiam} \rm(i) Let $\mu$ be a non-atomic measure. If $d_{\varphi_*} < b_{\varphi_*}$ and $\varphi_*(b_{\varphi_*})
\mu(\Omega) > 1$ then $L_\varphi^0$ does not have the local diameter two property.

\rm(ii) Let $\mu$ be the counting measure on $\mathbb{N}$. If $d_{\varphi_*} < c_{\varphi_*}$ and $\varphi_*(c_{\varphi_*}) =1$
 then $\ell_\varphi^0$ does not have the local diameter two property.
\end{Theorem}

\begin{proof}  We will show only (i), since the sequence case is proved similarly.  By the assumptions in view of Theorem \ref{th:KamKub}(i), the space $L_{\varphi_*}$ has a uniformly non-$\ell_1^2$  point. In view of Theorem \ref{th:unif}  it is equivalent to that the dual space $(L_{\varphi_*})^*$ does not have the weak$^*$-star local diameter two property.  It is well-known that the dual space to Orlicz space $L_\varphi$ is isometrically isomorphic to  the direct sum $L_{\varphi_*}^0 \oplus_1 \mathcal{S}$, where $\mathcal{S}$ is a set of the singular functionals on $L_\varphi$ (\cite{Chen}).

 Therefore  the dual space   $ (L_{\varphi_*})^*$ is  isometrically isomorphic to $L_\varphi^0 \oplus_1 \mathcal{S}$ due to $\varphi_{**} = \varphi$ \cite{KR}.  By Theorem \ref{th:KamKub}(i), there exists a uniformly non-$\ell_1^2$ point $x\in S_{L_{\varphi_*}}$ of a unit ball in $L_{\varphi_*}$.  Hence in view of Proposition \ref{th:unif}, there exists $\epsilon > 0$ such that ${\rm diam}\, S(x;\epsilon) < 2$ where $S(x;\epsilon) = \{x^* \in B_{(L_{\varphi_*})^*} : x^*(x) > 1 - \epsilon\}$ is a weak$^*$-slice. Now, let $J: L_{\varphi_*} \rightarrow (L_{\varphi_*})^{**}$ be the canonical mapping so that $J(x)(x^*) = x^*(x)$. Letting $i: L_{\varphi}^0 \rightarrow (L_{\varphi_*})^*$  be  isometric embedding,  $T: = J(x) \circ i\in B_{(L_\varphi^0)^*}$ and    $S(T;\epsilon) = 	\{y \in B_{L_{\varphi}^0} : T(y) > 1 - \epsilon\} $ is a slice of the unit ball in $L_\varphi^0$. Moreover,
	\[
S(T;\epsilon) \subset \{x^* \in B_{(L_{\varphi_*})^*} : J(x)(x^*) > 1 - \epsilon\} =  \{x^* \in B_{(L_{\varphi_*})^*} : x^*(x) > 1 - \epsilon\} = S(x; \epsilon).
	\]
	 Therefore, ${\rm diam} \,S(T;\epsilon)<2$, and the space $L_\varphi^0$ does not have the local diameter two property.
\end{proof}

In \cite{AKM2} it has been proved that if $\varphi$ does not satisfy the appropriate $\Delta_2$ condition then $L_\varphi$ or $\ell_\varphi$ has the local diameter two property.  This result was generalized later to Orlicz-Lorentz spaces \cite{KT}. For the Orlicz spaces equipped with the Orlicz norm the situation is different.  As shown below, for a  large class of finite Orlicz functions the spaces $L_\varphi^0$ or $\ell_\varphi^0$ have no local diameter two property.

\begin{Corollary}\label{Cor:Orlicznormdiam}
	Let $\mu$ be a non-atomic measure on $\Sigma$ or $\mu$ be the counting measure on $\mathbb{N}$.	Let $\varphi$ be a finite  $N$-function at infinity.  Then there exists a slice of  $B_{L_\varphi^0}$, respectively of $B_{\ell_\varphi^0}$, with diameter less than two. Consequently, the Orlicz spaces $L_\varphi^0$ or $\ell_\varphi^0$ equipped with the Orlicz norm do not have the local diameter two property.
\end{Corollary}

\begin{proof}
	Since $\varphi$ is an $N$-function at infinity then in view of Lemma \ref{lem:finite}, $\varphi_*$ is a  finite  function and so $b_{\varphi_*} = \infty$.  We also have that $d_{\varphi_*} <\infty$. Indeed, if for a contrary $d_{\varphi_*} =\infty$ then $\varphi_*(v)  = kv$ for some $k>0$ and all $v\ge 0$. Then it is easy to show that $\varphi = \varphi_{**}$ assumes only zero or infinity values,  which contradicts the assumption that $\varphi$ is a finite Orlicz function.    We complete the proof by
application of Theorem \ref{thm:Orlicznormdiam}.\end{proof}

We conclude this paper by showing that the SD2P, D2P and LD2P are equivalent in $L_\varphi$ or $\ell_\varphi$ when  $\varphi$ does not satisfy the appropriate $\Delta_2$ condition. Recall a subspace $Y$ of $X^*$ is said to be {\it norming} if for every $x \in X$,
\[
\|x\| = \sup\{|x^*(x)| : \|x^*\|_{X^*} \leq 1, x^* \in Y\}.
\]

\begin{Proposition}\cite[Proposition 1.b.18]{LT2}, \label{LT2}
	If $X$ is a Banach function space with the Fatou property, then the K\"othe dual space $X'$ is order isometric to a norming subspace of $X^*$.
\end{Proposition}

We say a closed subspace $Y$ is an \emph{$M$-ideal} in $X$ if $Y^\perp$ is the range of the bounded projection  $P: X^* \rightarrow X^*$ such that $\|x^*\| = \|Px^*\| + \|(I-P)x^*\|$, that is $X^* = Y^{\perp} \oplus_1 Z$ for some subspace $Z$ of $X^*$. In fact, there is a connection between $M$-ideals and the SD2P.

\begin{Theorem} \cite[Theorem 4.10]{ALN}\label{SD2P}
	Let $Y$ be a proper subspace of $X$ and let $Y$ be an $M$-ideal in $X$ i.e. $X^* = Y^{\perp} \oplus_1 Z$.   If $Z$ is  a norming subspace of $X^*$, then both $X$ and $Y$ have the strong diameter two property.
\end{Theorem}

\begin{Corollary}\label{th:Mideal}
Let $\mu$ be a non-atomic measure on $\Sigma$ or the counting measure on $\mathbb{N}$. Given a finite  Orlicz function $\varphi$ which does not satisfy the appropriate $\Delta_2$ condition, the spaces $L_{\varphi}$ or  $\ell_\varphi$ and their  proper subspaces $(L_{\varphi})_a\ne \{0\}$ or $(\ell_\varphi)_a\ne \{0\}$ have the strong diameter two property.

\end{Corollary}

\begin{proof}
Let $\mu$ be non-atomic. By the assumption that $\varphi$ is finite, the subspace $(L_{\varphi})_a$ is non-trivial. Moreover it is well-known that  it is an $M$-ideal in $L_{\varphi}$  \cite{HWW}. It is a proper subspace if $(L_{\varphi})_a\ne L_{\varphi}$, which is equivalent  to that $\varphi$ does not satisfy the appropriate $\Delta_2$ condition. By Proposition \ref{LT2}, $(L_{\varphi})' \simeq ((L_{\varphi})_a)^*$ is a norming subspace of $(L_\varphi)^*$. Hence by Theorem \ref{SD2P},  both $(L_{\varphi})_a$ and $L_{\varphi}$  have the strong diameter two property. The proof in sequence case is similar.
\end{proof}

The $M$-ideal property of the order continuous subspace of an Orlicz-Lorentz space has been studied \cite{KLT}. In our final result, we obtain full characterization of (local, strong) diameter two properties in Orlicz spaces equipped with the Luxemburg norm.  It is  completion and extension of Theorems 2.5 and 2.6 from \cite{AKM}, where it was shown that $L_\varphi$ or $\ell_\varphi$ have the D2P whenever $\varphi$ does not satisfy appropriate condition $\Delta_2$.

\begin{Theorem}\label{OReq} Let $\mu$ be a non-atomic measure on $\Sigma$ or the counting measure on $\mathbb{N}$ and let $\varphi$ be a finite Orlicz  function. Consider the following properties.
	\begin{itemize}
		\item[(i)]  $L_\varphi$ or $\ell_\varphi$ has the local diameter two property.
		\item[(ii)]  $L_\varphi$ or $\ell_\varphi$ has the  diameter two property.
		\item[(iii)]  $L_\varphi$ or $\ell_\varphi$ has the strong  diameter two property.
		\item[(iv)] $\varphi$ does not satisfy the appropriate $\Delta_2$ condition.
	\end{itemize}
Then $\rm(iii) \implies (ii) \implies (i)$. For the sequence space $\ell_\varphi$ all properties $\rm(i)-(iv)$ are equivalent. If in addition $\varphi$ is $N$-function at infinity then all $\rm(i)-(iv)$ are also equivalent for the function space $L_\varphi$.
\end{Theorem}

\begin{proof}
	The fact $\rm(iii) \implies (ii) \implies (i)$ is well-known in general Banach spaces \cite{ALN, GGMS}. The implication $\rm(iv) \implies (iii)$ follows from Corollary \ref{th:Mideal}. If $L_\varphi$ has the local diameter two property then the space can not have the RNP. Thus  from Theorems \ref{th:OrRN-funct} and \ref{th:RNP-ORseq},  (i) $\implies$ (iv).
\end{proof}

\end{document}